\documentclass[psamsfont,a4paper,12pt]{amsart}

\usepackage[english]{babel}   
\usepackage[hidelinks]{hyperref}

\usepackage{setspace}     
\usepackage{enumerate} 	
\usepackage{afterpage} 	
\usepackage[usenames,dvipsnames]{color}
\usepackage{graphicx}	

\usepackage{amsmath} 	
\usepackage{amssymb}	
\usepackage{mathrsfs} 	
\usepackage{amsfonts} 	
\usepackage{latexsym} 	
\usepackage[latin1]{inputenc} 

\usepackage{amsthm} 	
\usepackage{stackrel}       
\usepackage{amscd} 	

\usepackage{tabularx}	
\usepackage{longtable}	

\usepackage{lipsum}

\usepackage{xfrac}

\allowdisplaybreaks

\def\p{\mathfrak{p}}
\def\aa{\mathfrak{a}}

\def\cal{\mathcal}

\def\S{{\cal S}}

\def\O{{\cal O}}

\newcommand{\R}{\mathbb{R}}

\newcommand{\Z}{\mathbb{Z}} 
\newcommand{\Q}{\mathbb{Q}}

\renewcommand{\to}{\longrightarrow}

\def \inf{\mathop{\rm{inf}}}

\def\mod{\mathop{\rm{mod}}}

\newtheorem{Thm}{Theorem}[section]		
\newtheorem{Lemma}[Thm]{Lemma} 
\newtheorem*{remark}{Remark}

\newtheorem{Cor}[Thm]{Corollary}

\theoremstyle{definition}

\theoremstyle{remark}

\newtheorem{ind}[]{{\rm\it Indice}}

\usepackage{multicol}

\usepackage{tikz}
\usetikzlibrary{arrows}

\usepackage{paralist}
\usepackage{cite}

\usepackage{breqn}   

\title[Pair correlation for Dedekind zeta functions of abelian extensions]{Pair correlation for Dedekind zeta functions\\ of abelian extensions}
\author{David de Laat, Larry Rolen, Zack Tripp, Ian Wagner}

\begin{document}
\numberwithin{equation}{section}


\begin{abstract}
Here we study problems related to the proportions of zeros, especially simple and distinct zeros on the critical line, of Dedekind zeta functions. We obtain new bounds on a counting function that measures the discrepancy of the zeta functions from having all zeros simple. In particular, for quadratic number fields, we deduce that more than 45\% of the zeros are distinct. This extends work based on Montgomery's pair correlation approach for the Riemann zeta function. Our optimization problems can be interpreted as interpolants between the pair correlation bound for the Riemann zeta function and the Cohn-Elkies sphere packing bound in dimension $1$. We compute the bounds through optimization over Schwartz functions using semidefinite programming and also show how semidefinite programming can be used to optimize over functions with bounded support.
\end{abstract}

\maketitle

\section{Introduction}

In this paper, we study the zeros of Dedekind zeta functions on the critical line. To motivate the results and setup our main questions, we first recall 
the Riemann zeta function, defined as $\zeta(s) := \sum_{n=1}^\infty n^{-s}$ for $\mathrm{Re}(s)>1$. The Riemann Hypothesis (RH) states that the non-trivial zeros of $\zeta(s)$ are all on the line $\text{Re}(s) = \frac{1}{2}$. Furthermore, an important conjecture states that that all of its zeros are simple. Although this conjecture is far out of current reach, 
progress has been made on bounding measures of the discrepancy of these zeros from being simple.  
As usual, we denote by $N(T)$ the number of zeros $\rho=\beta+i\gamma$  of $\zeta(s)$ in the critical strip (counting multiplicity)  with  $0 < \gamma \le T$.
That is,  
\[
N(T):=\sum\limits_{\substack{0 < \gamma \le T }} 1. 
\]
We also define the counting functions
\begin{equation*}
N_s(T) := \sum\limits_{\substack{0 < \gamma \le T \\ m_\rho = 1}} 1,\quad\quad N_d(T) := \sum\limits_{0<\gamma \le T} \frac{1}{m_\rho},\quad\quad N^*(T) := \sum\limits_{0 < \gamma \le T} m_\rho,
\end{equation*}
which count the number of simple zeros, distinct zeros, and multiplicities of zeros respectively, where $m_\rho$ is the multiplicity of the zero $\rho = \beta + i \gamma$. Of course, assuming the simplicity conjecture one would have
\[
N(T)\stackrel{?}{=}N_s(T)\stackrel{?}{=}N_d(T)\stackrel{?}{=}N^*(T).
\]
Recently, Chirre, Gon\c calves, and the first author \cite{CGL} used semidefinite programming to obtain improved estimates of these quantities. This provides a method to bootstrap the asymptotic on the pair correlation function of the zeros of $\zeta(s)$ using numerical optimization. The study of the pair correlation function of the (normalized) zeros of
$\zeta(s)$ was pioneered by Montgomery\cite{M}, which has subsequently led to deep conjectural insights connecting the zeros of $\zeta(s)$ to random matrix theory and has led to a broad framework of $n$-point correlation functions and random matrix model predictions for $L$-functions by \cite{KaSa}, \cite{KeSn}, \cite{O}, and \cite{RS}, among many others. Utilizing Montgomery's asymptotic, the following new bounds on $N_d$ and $N^*$ were obtained in \cite{CGL}, conditional on RH:
\[ N_d(T) \ge (0.8477 +o(1))N(T) \quad\text{ and }\quad N^*(T) \le (1.3208 + o(1))N(T). \]
This improved the previously known bounds of $0.8466$ (from \cite{FGL}) and $1.3275$ (from \cite{CG}) respectively. We also recall that several other authors have studied the the proportions of distinct zeros of Dirichlet $L$-functions instead of $\zeta(s)$. For example, \cite{W} gives a lower bound of $0.8321$ of zeros in $q$-aspect of the entire family of Dirichlet $L$-functions being distinct. That paper makes use of the important Asymptotic Large Sieve, originating from the work of Conrey, Iwaniec, and Soundararajan (see \cite{CIS1,CIS2}).\newline 
\indent Given these results, it is natural to ask about analogous results for other zeta functions. Here we consider Dedekind zeta functions for abelian extensions, where we extend Montgomery's pair correlation method to give bounds for families of Dedekind of zeta functions. To be precise, if $K/\Q$ is a number field of degree $n$ with ring of integers $\O_K$, the Dedekind zeta function is 
\[ \zeta_K(s) := \sum\limits_{\aa\subseteq\O_K} \frac{1}{N(\aa)^s},  \]
where the sum is over ideals $\aa$ of $\O_K$ and $N(\aa)$ is the norm of $\aa$. For $K = \Q$, $\zeta_K(s)$ is simply the Riemann zeta function, while if $K=\Q(\sqrt D)$ is a quadratic extension,  $\zeta_K(s)$ factors as $\zeta_K(s)=\zeta(s)L(\chi,s)$, where $\chi$ is the quadratic character of $K$, and where 
\[
L(\chi, s) := \sum\limits_{n=1}^\infty \chi(n)n^{-s}
\]
for $\mathrm{Re}(s)>1$. More generally, by class field theory, for any abelian number field $K$, $\zeta_K(s)$ factors as $\zeta_K(s)=\zeta(s)\cdot\prod_{i=1}^{n-1}L(\chi_i,s)$ for some Dirichlet characters $\chi_i$. For any finite extension $K/\Q$, the Grand Riemann Hypothesis (GRH) predicts that the zeros of $\zeta_K(s)$ also lie on the line $\text{Re}(s) = \frac{1}{2}$. Finally, we let $N_K(T)$, $N_{K, d}(T)$, and $N_K^*(T)$ be defined exactly as $N(T)$, $N_d(T)$, and $N^*(T)$ above, but with the sums now being over the zeros of $\zeta_K(s)$. \newline 
\indent In this paper, we will restrict our focus to the case where $K/\Q$ is an abelian extension. Similarly to Montgomery \cite{M}, we define the pair correlation function of the zeros of $\zeta_K(s)$ to be 
\[ F_K(\alpha) := \left(\frac{nT}{2\pi} \log T \right)^{-1} \sum\limits_{0 < \gamma, \gamma' \le T} T^{i\alpha (\gamma-\gamma')}w(\gamma-\gamma'), \]
where $\alpha$ and $T\ge 2$ are real and $w(u) = 4/(4+u^2)$. Our first result to prove our main bounds is the following asymptotic formula.
\begin{Thm}\label{pair correlation}
Let $K$ be an abelian number field of degree $n$, and assume GRH. Then we have 
\[ F_K(\alpha) = (n+o_K(1))T^{-2|\alpha|}\log T + |\alpha| + o_K(1) \]
uniformly for $\lvert\alpha\rvert \le 1$ as $T \to \infty$. 
\end{Thm}
Using semidefinite programming, we are then able to obtain results analogous to those of \cite{CGL} for all abelian extensions. 
\begin{Thm}\label{SDP bound}
Assume GRH and the notation above. Then we have 
\[ N_K^*(T) \le (c_n + o_K(1)) N_K(T), \]
where
\[
c_n = \begin{cases}
2.3226 & \text{for $n=2$},\\
3.3232 & \text{for $n=3$},\\
4.3235 & \text{for $n=4$},\\
(1+10^{-10})n + 0.3243 & \text{for $n \ge 1$},\\
n + 1/3& \text{for $n \ge 1$}.
\end{cases}
\] 
\end{Thm}
The use of optimization techniques have proved to be useful elsewhere in number theory, such as for the study of prime gaps \cite{CMS} and spacings between zeros of $\zeta(s)$ \cite{CCLM}. While we were unable to find previous results on $N_K^*$ for general abelian extensions, work of Conrey, Ghosh, and Gonek \cite{CGG} gave that on GRH, at least $1/27$ of the zeros of $\zeta_K(s)$ are simple for quadratic extensions $K/\Q$. As a corollary we get new results on the proportion of zeros that are distinct.
\begin{Cor}\label{low degrees} Assuming GRH, we have the following.
\begin{enumerate}
\item[i).]
If $K/\Q$ is a degree 2 extension, then
\[ N_{K,d}(T) \ge (0.4585 + o_K(1))N_K(T). \]
\item[ii).]
If $K/\Q$ is a degree 3 extension, then 
\[ N_{K,d}(T) \ge (0.2794 + o_K(1))N_K(T). \]
\item[iii).]
If $K/\Q$ is a degree 4 extension, then 
\[ N_{K,d}(T) \ge (0.1127 + o_K(1))N_K(T). \]
\end{enumerate}
\end{Cor}
\begin{remark}
While there did not appear to be explicit results of this type in the literature previously, bounds of \cite{Ba}, \cite{CGL}, and \cite{CGG} could directly be combined to give an elementary estimate of about $0.3976$ for quadratic extensions. For higher degrees, our bounds on $N_K^*$ versus $n$ fail to produce anything new concerning distinct zeros.
\end{remark}

The paper is organized as follows.  In Section~\ref{Prelim1} we establish a few basic definitions and lemmas required for the proof of the pair correlation asymptotic of Theorem~\ref{pair correlation}, and in Section~\ref{Prelim2} we give a general description of the semidefinite programming techniques and how they can be used to obtain bounds 
on quantities in analytic number theory. We conclude with the proofs of the theorems in Section~\ref{ProofsS}.

\section*{Acknowledgments}

The authors thank Ken Ono for useful discussions related to this work.

\section{Preliminaries}\label{PrelimSec}

In this section, we review the basic definitions and notations for the proof of Theorem~\ref{pair correlation}. Since almost every estimate depends on the field $K$ in some way, we will drop it from the subscript for our big-$O$ estimates in this section and throughout the rest of the paper. 

\subsection{Ingredients for the proof of Theroem~\ref{pair correlation}}\label{Prelim1}

For $\zeta_K(s)$, the analogue of the Riemann $\xi$ function is 
\begin{equation}\label{xi_K}
\xi_{K}(s) := \frac{1}{2}s(s-1) |\Delta_{K}|^{\frac{s}{2}} \pi^{-\frac{ns}{2}} 2^{(1-s)r_{2}}\Gamma\left(\frac{s}{2}\right)^{r_{1}} \Gamma(s)^{r_{2}} \zeta_{K}(s),
\end{equation}
where $r_1$ and $r_2$ are the number of real embeddings and the number of pairs of complex embeddings of $K$ respectively, yielding the relation $r_1 + 2r_2 = n$, and $\Delta_K$ is the discriminant of $K$. As in the Riemann case, $\xi_K$ is entire, shares the same non-trivial zeros as $\zeta_K$, and has the functional equation $\xi_K(s) = \xi_K(1-s)$ (see \cite{N}). Because of this functional equation, we call the region $0 < \text{Re}(s) < 1$ the \textit{critical strip}. Using the functional equation and properties of the $\Gamma$ function, it is easy to check that there are trivial zeros of $\zeta_K$; namely, there are zeros of order $r_1 + r_2$ at negative even integers, of order $r_2$ at negative odd integers, and of order $r_1 + r_2 - 1$ at $0$. Moreover, it is possible to show that these are the only zeros not in the interior of the critical strip. We also define the following generalization of the von Mangoldt function:
\[ \Lambda_K (\mathfrak{a}) = \left\{
\begin{array}{ll}
\log N(\mathfrak{p})  & \text{if } \mathfrak{a} = \mathfrak{p}^k \text{ for } \mathfrak{p} \text{ prime}\\
0  & \text{otherwise}.
\end{array} \right.\]
Note that $\Lambda_{\mathbb{Q}}=\Lambda$ is the classical von Mangoldt function.
Similar to how the von Mangoldt function gives the coefficients of $-\zeta'/\zeta$, one can easily see that
\begin{equation}\label{log deriv}
-\frac{\zeta_K'}{\zeta_K}(s) = \sum\limits_{\mathfrak{a}} \frac{\Lambda_K(\mathfrak{a})}{N(\mathfrak{a})^s},
\end{equation}
where the sum is over ideals $\aa$ of $\O_K$. Using this fact, we can obtain the following formula.
\begin{Lemma}\label{sum_over_zeros}
	For $x > 1$ and $s \neq 1, 0, -m, \rho$, we have
	\[\sideset{}{'}\sum\limits_{N(\mathfrak{a})\le x}^{}\frac{\Lambda_K (\mathfrak{a})}{N(\mathfrak{a})^s} = -\frac{\zeta_K'}{\zeta_K}(s) + \frac{x^{1-s}}{1-s} - \sum\limits_{\rho} \frac{x^{\rho - s}}{\rho - s} + r_1 \sum\limits_{m=0}^\infty \frac{x^{-2m-s}}{2m+s} + r_2\sum\limits_{m=0}^\infty \frac{x^{-m-s}}{m+s} - \frac{x^{-s}}{s}, \]
	where the sum is over the zeros $\rho$ of the $\zeta_K$.
\end{Lemma}
In this lemma and throughout the rest of the paper, $\sideset{}{'}\sum_{N(\aa) \le x}$ will indicate that terms with $N(\aa) = x$ are multiplied by $1/2$ in the sum. This is analogous to existing results for $\zeta(s)$ and $L(\chi, s)$, which can be found in \cite{L} and \cite{Y} respectively. We will omit the proof since it is easily adapted from previous proofs. For example, by replacing $J(x,T)$ by
\[ J(x,s,T) = \frac{1}{2\pi i}\int\limits_{c-iT}^{c+iT} \left[ -\frac{\zeta_K'(s)}{\zeta_K(s)} \right] \frac{x^{z-s}}{z-s} dz \]
for $0 < s < 1$ in Chapter 17 of \cite{D} and using the additional estimate that $a_K(m):= \{\aa : N(\aa) = m\} \ll m$, all of the details work essentially the same way as in Davenport, and by uniqueness of analytic continuation, the lemma then holds for the desired $s$. The only other piece of information we need for Davenport's proof to work is that $|\zeta_K'(s)/\zeta_K(s)| \ll \log(2|s|)$ for $\sigma \le -1$ and bounded away from the negative integers, which follows from a nonsymmetric functional equation for $\zeta_K(s)$. We will prove this now and use this version of the functional equation for the proof of the next lemma as well. \newline 
\indent Using the functional equation $\xi_K(s) = \xi_K(1-s)$ and the definition of $\xi_K$ given in (\ref{xi_K}), we can solve for $\zeta_K(1-s)$ and see that
\begin{equation}\label{func eq}
\zeta_K(1 - s) = |\Delta_K|^{s - \frac{1}{2}}\pi^{n(\frac{1}{2}-s)} 2^{(1-2s)r_2} \frac{\Gamma(\frac{s}{2})^{r_1} \Gamma(s)^{r_2}}{\Gamma(\frac{1-s}{2})^{r_1}\Gamma(1-s)^{r_2}} \zeta_K(s).
\end{equation}
From 
\[ \frac{\Gamma(\frac{s}{2})}{\Gamma(\frac{1-s}{2})} = \pi^{-\frac{1}{2}}2^{1-s} \cos \frac{s\pi}{2} \Gamma(s) \text{ and } \frac{\Gamma(s)}{\Gamma(1-s)} = \frac{\Gamma(s)^2 \sin \pi s}{\pi}, \]
(see Chapter 10 of \cite{D}) we can write the logarithmic derivative of (\ref{func eq}) as
\begin{equation}\label{nonsym func eq}
-\frac{\zeta_K'}{\zeta_K}(1-s) = O(1) - \frac{\pi r_1}{2}\tan \frac{s\pi}{2} + \pi r_2 \cot \pi s + n \frac{\Gamma'}{\Gamma}(s) + \frac{\zeta_K'}{\zeta_K}(s).
\end{equation}
For $\sigma \ge 2$ (which corresponds to $1-\sigma \le -1$), $\zeta_K'(s)/\zeta_K(s)$ is bounded, and bounded away from their poles (which occur at integer values), $\tan(\pi s/2)$ and $\cot(\pi s)$ are bounded as well. From Stirling's asymptotic formula, we conclude that $\zeta_K'(s)/\zeta_K(s) \ll \log(2|s|)$ as desired. 
\indent From Lemma~\ref{sum_over_zeros} and the (\ref{nonsym func eq}), we will be able to prove the following lemma, which is analogous to a lemma of Montgomery \cite{M}. 
\begin{Lemma}\label{montgomery_lemma}
	Assume GRH. For $x \ge 1$ and $1 < \sigma < 2$ fixed, 
	\begin{align*}
	&(2\sigma - 1) \sum\limits_\gamma \frac{x^{i\gamma}}{(\sigma - \frac{1}{2})^2 + (t-\gamma)^2}\\
	&\indent= -x^{-\frac{1}{2}}\left(\sum\limits_{N(\aa) \le x} \Lambda_K(\mathfrak{a}) \left(\frac{x}{N(\aa)}\right)^{1-\sigma + it} + \sum\limits_{N(\aa) > x} \Lambda_K(\mathfrak{a}) \left(\frac{x}{N(\aa)}\right)^{\sigma + it} \right) \\
	&\indent\indent \indent + x^{\frac{1}{2}-\sigma + it}(n\log \tau + O_\sigma(1)) + O_\sigma(x^{\frac{1}{2}}\tau^{-2}) + O_\sigma (x^{-\frac{1}{2}}\tau^{-1}),
	\end{align*}
	where $\tau = |t| + 2$ and the sum is over the ordinates $\gamma$ of the zeros of $\zeta_K$.
\end{Lemma}
Since the proof of this is also a direct adaptation of that of Montgomery's, we will simply recall the outline of the proof. Subtract the formula of Lemma~\ref{sum_over_zeros} at $s = 1 - \sigma + it$ from the same formula at $s = \sigma + it$, multiply both sides by $x^{it}$, and use (\ref{log deriv}). This will yield all of the terms above except for the $x^{\frac{1}{2}-\sigma + it}$. This term comes from the nonsymmetric functional equation (\ref{nonsym func eq}), Stirling's approximation again, and the fact that the remaining terms in (\ref{nonsym func eq}) are bounded for a fixed $1 < \sigma < 2$. \newline 
\indent Now, we wish to explicitly find the Dirichlet coefficients of $-\zeta_K'(s)/\zeta_K(s)$. From (\ref{log deriv}), it is clear that they are given by $c_K(m):= \sum\limits_{N(\aa) = m} \Lambda_K(\aa)$. We will use the Euler product for $\zeta_K(s)$ and for Dirichlet $L$-functions to come up with a more explicit formula for $c_K(m)$.The Euler product for the Dedekind zeta function
\[ \zeta_K(s) = \prod\limits_{\p} \left( 1 - \frac{1}{N(\p)^s} \right)^{-1}, \]
where the product is over the prime ideals $\p$ of $\cal{O}_K$. If we assume that $K/\Q$ is a Galois extension, then we can rewrite this as a product over primes $p$ of $\Z$. If $K/\Q$ is Galois, we know that every prime $p$ in $\Z$ has a factorization of the form
\[ p\cal{O}_K = \p_1^{e} \dots \p_g^{e}, \]
where $e\ge 1$, the $\p_i$ are distinct primes of norm $p^f$ in $\O_K$, and $efg = n$. Using this, the Euler product becomes 
\[ \zeta_K(s) = \prod\limits_p (1 - p^{-fs})^{-g}. \]
Taking the logarithmic derivative of both sides and utilizing the Taylor series for $\log(1-x)$ gives
\begin{equation}\label{Dir1}
\frac{\zeta_K'}{\zeta_K}(s) = - \sum\limits_p \frac{n}{e}\log p \sum\limits_{k=1}^\infty \left(p^{kf}\right)^{-s}.
\end{equation}
From this, we could explicitly write down the $c_K(m)$, but we can obtain a more useful description in the abelian case. Recall that for an abelian extension $K/\Q$, we can write $\zeta_K$ as a product of Dirichlet $L$-functions 
\begin{equation}\label{prod of Ls}
\zeta_K(s) = \prod_{i=0}^{n-1}L(s, \chi_i),
\end{equation}
where $\chi_i$ is a Dirichlet character of conductor $q_i$, $\chi_0$ is the trivial character, and the product of the conductors is $\Delta_K$. Each Dirichlet $L$-function also has an Euler product, and through the same method as above, we can obtain a Dirichlet series for its logarithmic derivative:
\[ \frac{L'(s, \chi_i)}{L(s, \chi_i)} = - \sum\limits_p \log p \sum\limits_{k=1}^\infty \chi_i(p^k)p^{-ks}. \]
(This can also be found in \cite{D}.) We can now use this equation when taking the logarithmic derivative of (\ref{prod of Ls}) to find an another expression for the Dirichlet series
\begin{equation}\label{Dir2}
\frac{\zeta_K'}{\zeta_K}(s) = \sum\limits_{i=0}^{n-1} \frac{L'(s, \chi_i)}{L(s,\chi_i)} = - \sum\limits_p \log p \sum\limits_{k=1}^\infty \left(\sum\limits_{i=0}^{n-1} \chi_i (p^k)\right) p^{-ks}.
\end{equation}
From this, we will write down the necessary information about $c_K(m)$ in two separate cases. \newline 
\indent When $(m, \Delta_K) = 1$, we wish to show that 
\begin{equation}\label{c_K rel prime}
c_K(m) = \begin{cases} 
n \Lambda(m) & \text{if } \chi_0(m) = \chi_1(m) = \dots = \chi_{n-1}(m) = 1,\\
0 & \text{otherwise}.
\end{cases}
\end{equation}
If $m$ is not a prime power, it is clear from (\ref{Dir1}) that $c_K(m) = 0 = n\Lambda(m)$, so suppose $m = p^k$. If $\chi_i (p^k) = 1$ for all $i$, (\ref{Dir2}) clearly gives a coefficient of $n \log p = n\Lambda(m)$. On the other hand, it suffices to show that if $c_K(p^k) \neq 0$, then $\chi_i(p^k) = 1$ for all $i$. Note that $p$ being relatively prime to $\Delta_K$ means $p$ is unramified, i.e. $e = 1$. Therefore, (\ref{Dir1}) tells us that any non-zero coefficient of $p^{-ks}$ must be $n \log p$. But from (\ref{Dir2}), this can only occur if $\chi_i(p^k) = 1$ for all $i$, proving (\ref{c_K rel prime}).\newline
\indent On the other hand, when $(m, \Delta_K)> 1$, we will only need the fact that
\begin{equation}\label{c_K common factor}
0 \le c_K(m) \le n\Lambda(m),
\end{equation}
which clearly follows from (\ref{Dir1}). \newline 
\indent Finally, the last things we will need are sum estimates similar to those used by Montgomery. 
\begin{Lemma} \label{first_est}
	Let $(a, q) = 1$. Then
	\[\sum\limits_{\substack{m \le x \\ m \equiv a (\mod q)}} m\Lambda(m)^2 = \frac{1}{2\varphi(q)}x^2 \log x + O(x^2) \]
	for $x \ge q$.
\end{Lemma}
\begin{Lemma}\label{second_est}
	Let $(a, q) = 1$. Then
	\[\sum\limits_{\substack{m \le x \\ m \equiv a (\mod q)}} m^2\Lambda(m)^2 = O(x^3 \log x). \]
	for $x \ge q$.
\end{Lemma}
We will simply state that these estimates come from the prime number theorem on arithmetic progressions (see Chapter 20 of \cite{D}), where in the following $\varphi(q)$ denotes Euler's totient function.
\begin{Lemma}\label{greater_than}
	For $s > 1$ real and $(a,q)=1$, we have 
	\[\sum\limits_{\substack{m > x \\ m \equiv a (\mod q)}} \frac{\Lambda(m)}{m^s} = \frac{1}{\varphi(q)} \frac{x^{1-s}}{s-1} + O_{s,q}\left(x^{\frac{1}{2}-s}\log^2x\right).\]
\end{Lemma}
\begin{Lemma}\label{third_est}
	For $s > 1$ real and $(a,q)=1$, we have 
	\[\sum\limits_{\substack{m > x \\ m \equiv a (\mod q)}} \frac{\Lambda(m)^2}{m^s} = \frac{1}{\varphi(q)} \frac{x^{1-s}\log x}{s-1} + O_{s,q}\left(x^{1-s}\right).\]
\end{Lemma}
It is easy to prove Lemma~\ref{third_est} from Lemma~\ref{greater_than}. The proof of Lemma~\ref{greater_than} is similar to the proof of the same sum without the congruence conditions; however, the explicit formula for $\psi(x) = \sum_{n \le x} \Lambda(n)$ is replaced by the explicit formulas for $\psi(x, \chi) = \sum_{n \le x} \Lambda(n)\chi(n)$ for $\chi (\mod q)$, which can be found in \cite{Y} as mentioned before. While Yildirim's formula only holds for primitive characters, one can easily keep track of the constants in the difference between $L(s, \chi)$ and $L(s, \chi_1)$ when $\chi_1$ is a primitive character inducing $\chi$. Combining all of these formulas and using them as in the proof of the same sum without congruence conditions, one can deduced Lemma~\ref{greater_than} and hence Lemma~\ref{third_est} as well.

\subsection{Overview of semidefinite programming techniques}\label{Prelim2}

In this paper we consider linear programming problems of the form
\begin{align}\label{eq:lp}
\inf \{ L(f) : \;&f \in L^1(\R) \text{ even and continuous}, \\\nonumber
& f(x) \leq 0 \text{ for }  \lvert x \rvert \ge 1, \, \hat f(0) = 1, \,\hat f \ge 0\},
\end{align}
where $L$ is a linear functional. For $L(f) = f(0)$ this is the Cohn-Elkies sphere packing bound in dimension $1$ \cite{CE}.  The difficulty in solving this optimization problem is that is that we simultaneously consider constraints on $f$ and its Fourier transform. 

To find numerical approximations of the optimal solution, Cohn and Elkies parameterize the functions as
\begin{equation}\label{eq:par}
f(x) = p(x^2) e^{-\pi x^2},
\end{equation}
where $p$ is a polynomial of given degree $2d$. 
One approach is to define $f$ and $\hat f$ by their real roots (assuming there are sufficiently many, and that we know their degrees) and find good locations for the roots, which works very well for sphere packing problems where the correct root locations are known \cite{CM}. Here instead we will use semidefinite programming to optimize over functions of the form \eqref{eq:par} as was also done for the zeta function in \cite{CGL}. 

Let $\mathcal T$ be the linear operator so that $(\mathcal Tp)(x^2) e^{-\pi x^2}$ is the Fourier transform of $p(x^2) e^{-\pi x^2}$. Since $e^{-\pi x^2}$ is nonnegative, the constraints in \eqref{eq:lp} reduce to the constraints $p(x) \le 0$ for $x \ge 1$, $\mathcal T(p)(x) \ge 1$ for $x \ge 0$, and $\mathcal T(p)(0) = 1$. As explained in \cite{PS}, the first two constraints are equivalent to the condition that there exist sum-of-square polynomials $s_1,\ldots,s_4$ of degree at most $2d$ such that 
\[
p(x) = -s_1(x) - (x-1) s_2(x)\quad \text{and} \quad \mathcal T(p)(x) = s_3(x) + x s_4(x).
\] 

A polynomial $s(x)$ of degree $2d$ is a sum-of-squares polynomial if and only if it can be written as $s(x) = v(x)^{\sf T} Q v(x)$, where $v(x)$ is a vector whose entries form as basis of the polynomials up to degree $d$, and where $Q$ is a positive semidefinite matrix (a symmetric matrix with nonnegative eigenvalues).

This shows that if we restrict to functions of the form \eqref{eq:par} and set $s_i(x) = v(x)^{\sf T} Q_i v(x)$, then \eqref{eq:lp} reduces to
\begin{equation}\label{eq:sdp}
\inf \Big\{L(f) :  Q_1,\ldots,Q_4 \succeq 0,\, \mathcal T(p)(x) = s_3(x) + x s_4(x), \, s_3(0) = 1\Big\},
\end{equation}
where we use the notation $Q_1,\ldots,Q_4 \succeq 0$ to indicate that these are positive semidefinite matrices.

The constraints $\mathcal T(p)(x) = s_3(x) + x s_4(x)$ and $s_3(0) = 1$ are linear in the  entries of the matrices. If we can also write $L(f)$ as a linear functional in the entries of the matrices, then \eqref{eq:sdp} is a semidefinite program, which can be solved numerically using an interior point solver such as SDPA-GMP \cite{Na}. 

\section{Proofs of the main results}\label{ProofsS}

\subsection{Proof of Theorem \ref{pair correlation}}

We will first follow the outline of Montgomery's proof of the pair correlation asymptotics that allow us to obtain the formula for $F_K(\alpha)$ for $0 \le \alpha < 1$. Then we will follow it with the ideas of Goldston's proof from his thesis \cite{G} that allow us to obtain the formula uniformly for $0 \le \alpha \le 1$, which will be necessary for the proof of Theorem~\ref{SDP bound}. \newline
\indent Letting $\sigma = 3/2$ in Lemma~\ref{montgomery_lemma} and letting $L(x,t)$ and $R(x,t)$ be the left-hand and right-hand sides of the equation respectively, we wish to estimate $\int\limits_0^T |L(x,t)|^2 dt$ and $\int\limits_0^T |R(x,t)|^2 dt$ to obtain the desired asymptotic. First, define
\[ F_K(x,T) := \sum\limits_{0 < \gamma, \gamma' \le T} x^{i(\gamma - \gamma')}w(\gamma - \gamma') \]
for $x \ge 1$ and $T \ge 2$ real so that 
\begin{equation}\label{one_var}
F_K(\alpha) = \left(\frac{nT}{2\pi}\log T\right)^{-1} F(T^\alpha, T).
\end{equation}
Using residue calculus, one can see that 
\[ \int\limits_{-\infty}^\infty \frac{dt}{(1+(t-\gamma)^2)(1+(t-\gamma')^2)} = \frac{\pi}{2}w(\gamma - \gamma'), \]
so we can rewrite
\begin{equation}\label{two_var}
F_K(x,T) = \frac{2}{\pi} \int\limits_{-\infty}^\infty \left| \sum\limits_{0 < \gamma \le T} \frac{x^{i\gamma}}{1+(t-\gamma)^2}\right|^2 dt.
\end{equation}
As in Montgomery's proof, we can take the absolute value squared, switch sum and integral, and bound the sum of the integrals from $T$ to $\infty$ and the integrals from $-\infty$ to $0$ by $O(\log^2 T)$ using the fact that there are at most $\ll \log T$ zeros with ordinate $T \le \gamma \le T+1$ (see Chapter 5 of \cite{IK}). The same fact and same type of estimates allow us to bound the sums over $\gamma$ and $\gamma'$ outside of the interval $(0, T]$ by $O(\log^3 T)$, so combining all of this yields 
\begin{equation}\label{F estimate}
F_K(x,T) = \frac{2}{\pi} \int\limits_0^T \left| \sum\limits_\gamma \frac{x^{i\gamma}}{1 + (t-\gamma)^2} \right|^2 dt + O(\log^3 T) = \frac{1}{2\pi} \int\limits_0^T \left|L(x,t)\right|^2 dt + O(\log^3 T).
\end{equation}
Now, we use $L(x,t) = R(x,t)$. In order to estimate $\int\limits_0^T |R(x,t)|^2 dt$, we can use Cauchy-Schwarz and Parseval's identity: 
\begin{equation}\label{parseval}
\int\limits_0^T \left| \sum_m a_m m^{-it}\right|^2 dt = \sum_m |a_m|^2 (T + O(m)).
\end{equation}
Using this, we find that 
\begin{align}\label{int}
\frac{1}{x} \int\limits_0^T \left|\left(\sum\limits_{N(\aa) \le x} \Lambda_K(\mathfrak{a})\right.\right.&\left.\left. \left(\frac{x}{N(\aa)}\right)^{-\frac{1}{2} + it} + \sum\limits_{N(\aa) > x} \Lambda_K(\mathfrak{a}) \left(\frac{x}{N(\aa)}\right)^{\frac{3}{2} + it} \right)\right|^2 dt \nonumber \\
&= \frac{1}{x}\sum_{m \le x} c_K(m)^2 \left( \frac{x}{m}\right)^{-1} (T+O(m)) + \frac{1}{x}\sum_{m > x} c_K(m)^2 \left(\frac{x}{m}\right)^3 (T+O(m)).
\end{align}
First, we can break up the sum into sums over relatively prime congruence classes modulo $\Delta_K$. Let $C_K$ be the number of congruence classes $m$ modulo $\Delta_K$ such that $\chi_i(m) = 1$ for $i = 0, \dots, n-1$. For this part of the sum, we combine equation (\ref{c_K rel prime}), Lemma \ref{first_est}, Lemma \ref{second_est}, and Lemma \ref{third_est} to see that this last expression is equal to
\begin{align}\label{term_zero}
\frac{n^2 T}{x^2}\cdot &C_K \left[ \frac{1}{2\varphi(\Delta_K)} x^2 \log x + O(x^2)\right] + O(x\log x) \nonumber \\  
&\indent \indent + n^2 Tx^2 \cdot C_K\left[ \frac{1}{2\varphi(\Delta_K)} \frac{\log x}{x^2} + O(x^{-2})\right] + O(x\log x) \nonumber \\
&= T\left(\frac{n^2 C_K}{\varphi(\Delta_K)} \log x + O(1)\right) + O(x \log x).
\end{align}
We can simplify this expression further by considering the proportion of primes that split completely in $K$. By Chebotarev density theorem, we know these primes have density $1/n$ in the set of all primes. On the other hand, since the set of these primes is the union of $C_K$ congruence classes modulo $\Delta_K$, we see that they also have density $C_K/\varphi(\Delta_K)$ in the set of all primes, so $C_K/\varphi(\Delta_K) = 1/n$. Therefore, (\ref{term_zero}) becomes
\begin{equation}\label{term_one}
T\left( n\log x + O(1)\right) + O(x \log x) =: M_1.
\end{equation}
Note that using (\ref{c_K common factor}), we can group the terms with $(m,\Delta_K) > 1$ into the error terms by directly computing the sum over the powers of a fixed prime, so (\ref{int}) is actually equal to (\ref{term_one}). It is easy to see that the integrals of the squares of the remaining terms for $R(x,t)$ are given by
\begin{align*}\label{term_two}
M_2 := \int\limits_0^T |nx^{-1 + it}\log \tau|^2 dt &= \frac{n^2 T}{x^2}(\log^2 T + O(\log T)),\\
M_3 := \int\limits_0^T |O(x^{-1+it})|^2 dt &\ll \frac{T}{x^2},\\
M_4 := \int\limits_0^T |O(x^{\frac{1}{2}}\tau^{-1})|^2 dt &\ll x.
\end{align*}
For $1 \le x \le (\log T)^{3/4}$, $M_i = o(M_2)$ for $i = 1,3,4$. For $(\log T)^{3/4} < x \le (\log T)^{3/2}$, all of the terms are uniformly $o(T \log T)$. For $(\log T)^{3/2}<x\le T/\log T$, $M_i = o(M_1)$ for $i =2,3,4$. Combining all of this using Cauchy-Schwarz, combining (\ref{one_var}) and (\ref{two_var}), and plugging in $x = T^\alpha$, we get
\[ nF_K(\alpha)T\log T + O(\log^3 T) = \int\limits_0^T |R(T^\alpha, t)|^2 dt = \left((1+ o(1))n^2 T^{-2\alpha}\log T + n\alpha + o(1)\right)T\log T \]
uniformly in $0 \le \alpha \le 1 - \varepsilon$. \newline 
\indent To obtain the uniformity near $1$, we refine the estimate (\ref{F estimate}). The sum inside the absolute value of the integrand of this equation is bounded by $\ll \log \tau$ (Chapter 5 of \cite{IK}), so the integral from $0$ to $T^{1/2}$ is seen to be bounded above by $\ll T^{1/2}\log^2 T \ll T$. Then in order to rewrite the integral from $T^{1/2}$ to $T$, we can repeat all of the steps above, namely to use Lemma~\ref{montgomery_lemma} with $\sigma = 3/2$ and to use Cauchy-Schwarz. In this case though, we compare the terms of the equation for $T^\varepsilon \le x \le T^2$ instead, in which case the error can be shown to be bounded above by $O(T)$. In other words, using the same type of estimates as above yields
\begin{equation}\label{F rewrite}
F_K(x,T) = \frac{n^2}{2\pi x}\int\limits_{T^{1/2}}^T \left| \sum\limits_{m \in \S} \frac{\Lambda(m)d_m(x)}{m^{it}} \right|^2 dt + O(T)
\end{equation}
for $T^\varepsilon \le x \le T^2$, where $d_m(x) = \min((x/m)^{-1/2}, (x/m)^{3/2})$ and $\S$ is the set of natural numbers $m$ with $\chi_i(m) = 1$ for $i = 0, \dots, n-1$. As Goldston does, we define two new auxiliary functions in order to evaluate this integral:
\begin{align*}
A(x,T)&:= \frac{1}{x}\int\limits_{0}^T \left| \sum\limits_{m \in \S} \frac{\Lambda(m)d_m(x)}{m^{it}} \right|^2 dt \\
B(x,T)&:= \frac{1}{x}\int\limits_{-T}^T \left( 1 - \frac{|t|}{T}\right) \left| \sum\limits_{m \in \S} \frac{\Lambda(m)d_m(x)}{m^{it}} \right|^2 dt.
\end{align*}
While $A(x,T)$ is the desired integral we wish to compute to further write down $F_K(x,T)$, we can square the absolute value in the integral of $B(x,T)$, which yields
\begin{equation*}
B(x,T) = \frac{T}{x} \sum\limits_{m \in \S} \Lambda^2(m)d_m^2(x) + \frac{T}{x} \sum\limits_{\substack{m, j \in \S \\ m \neq j}} \Lambda(m)\Lambda(j)d_m(x)d_j(x)\left(\frac{\sin\left(\frac{T}{2}\log \frac{m}{j}\right)}{\frac{T}{2}\log \frac{m}{j}}\right)^2 =: S_1 + S_2.
\end{equation*}
Through the same calculations as above using the sum estimates of section~\ref{Prelim1}, we determine that $S_1 = (T/n)\log x + O(T)$. For $S_2$, notice that regardless of the values of $m$ and $j$, the terms of the sum are non-negative. Thus, $S_2$ is at most the sum over all $m\neq j$, i.e. at most the sum without the character conditions. In Goldston's thesis, he shows that this is $O(x)$ (see pages 48-53 of \cite{G}). Combined, this yields
\[ B(x,T) = \frac{1}{n}T \log x + O(T) + O(x). \]
Now, in order to compute $A(x,T)$, the following relations are easily obtained from writing down the integral definition of $B(x,T)$:
\[ TB(x,T) - (T-\delta) B(x,T-\delta) \le 2\delta A(x,T) \le (T+\delta)B(x,T+\delta)-TB(x,T) \]
for any $\delta > 0$. From the formula we derived for $B(x,T)$, this tells us that
\[ \left| A(x,T) - \frac{1}{n}T\log x \right| \le \frac{\delta}{2n}\log x + O(x) + O(T) + O(\delta) + O\left( \frac{T}{\delta}(x+T) \right). \]
With the appropriate choice of $\delta$, namely $\delta = T$ for $1 \le x \le 2$, $\delta = T/(\log x)^{1/2}$ for $2 < x \le T$, and $\delta = (Tx/\log x)^{1/2}$ for $x >T$, we can rewrite the error terms to obtain 
\[ A(x,T) = \frac{1}{n} T\log x + O(x) + O(T) + O\left( (xT\log x)^{\frac{1}{2}} \right) + O\left( T(\log x)^{\frac{1}{2}} \right). \]
By (\ref{F rewrite}), we obtain
\begin{align*}
F_K(x,T) &= \frac{n^2}{2\pi} \left(A(x,T) - A(x, T^{\frac{1}{2}})\right) + O(T)\\
&= \frac{n}{2\pi}T\log x + O(x) + O\left( (xT\log x)^{\frac{1}{2}} \right) + O\left( T(\log x)^{\frac{1}{2}} \right)
\end{align*}
for $T^\varepsilon \le x \le T^2$. Plugging in $T^\alpha$ and solving for $F_K(\alpha)$ yields the desired result, noting that each of the error terms is $o(1)$ when $\varepsilon \le \alpha \le 1$.

\subsection{Proofs of Theorem~\ref{SDP bound} and Corollary~\ref{low degrees}}

Let $\mathcal{A}_{LP}$ be the set of  even continuous functions $f \in L^1(\mathbb{R})$ that satisfy:
\begin{enumerate}
	\item $\hat{f}(0) = 1$,
	\item $\hat{f} \ge 0$,
	\item $f(x) \leq 0$ for $\lvert x\rvert \ge 1$.
\end{enumerate}
It is known  (see Theorem 5.8 of \cite{IK}) that
\[ N_K(T) \sim \frac{nT}{2\pi} \log T. \]
Using Fourier inversion on the definition of $F_K(\alpha)$ and using this asymptotic, we obtain 
\begin{align}\label{fourier_inv}
\sum\limits_{0< \gamma, \gamma' \le T} g\left( (\gamma - \gamma')\frac{\log T}{2\pi}\right)w(\gamma - \gamma') &= \left(\frac{nT}{2\pi}\log T\right)\int\limits_{-\infty}^\infty \hat{g}(\alpha)F_K(\alpha) d\alpha \\ \nonumber
&= N_K(T)(1+o_K(1))\int\limits_{-\infty}^\infty \hat{g}(\alpha)F_K(\alpha) d\alpha 
\end{align}
for suitable $g$. To state the lemma, we define the following linear functionals for $f \in \mathcal{A}_{LP}$:
\[ \mathcal{Z}_n(f) = nf(0) + 2 \int\limits_0^1 f(x) x\, dx. \]

\begin{Lemma}\label{boundlemma}
	Let $f \in \mathcal{A}_{LP}$. Assuming GRH, we have 
	\[ N_K^*(T) \le (\mathcal{Z}_n(f) + o_K(1)) N_K(T). \]
\end{Lemma}

\begin{proof}
	By our theorems above, we have that
	\[ F_K(\alpha, T) = \left(nT^{-2|\alpha|}\log T + |\alpha|\right) (1 + o_K(1)) \]
	uniformly for $\lvert\alpha\rvert \le 1$. Since $f$ is continuous and $T^{-2|\alpha|}\log T\to \delta_0(x)$ as $T \to \infty$ in the distributional sense, we can rewrite (\ref{fourier_inv}) as
	\begin{align*}
	\sum\limits_{0 < \gamma,\gamma'\le T} \hat f &\left( (\gamma - \gamma')\frac{\log T}{2\pi} \right) w(\gamma - \gamma') \\
	&= N_K(T)\left[ n f(0) + \int\limits_{-1}^1 f(\alpha)\lvert\alpha\rvert d\alpha + \int\limits_{\lvert \alpha\rvert >1} f(\alpha)F_K(\alpha)d\alpha + o_K(1)\right].
	\end{align*}
	Since $F_K(\alpha)$ is non-negative, $f(x)\le 0$ for $\lvert x\rvert \ge 1$, and $f$ is even, we see that
	\begin{align*}
	\sum\limits_{0 < \gamma,\gamma'\le T} \hat f\left( (\gamma - \gamma')\frac{\log T}{2\pi} \right) w(\gamma - \gamma') &\le N_K(T)\left[ n f(0) + 2\int\limits_{0}^1 f(\alpha)\alpha d\alpha  + o_K(1)\right]\\
	&= N_K(T)\left[ \mathcal{Z}_n(f) + o_K(1)\right].
	\end{align*}
	On the other hand,
	\[ \sum\limits_{0 < \gamma,\gamma'\le T} \hat f\left( (\gamma - \gamma')\frac{\log T}{2\pi} \right) w(\gamma - \gamma') \ge \hat f(0) \sum\limits_{0 < \gamma \le T} m_\rho = N_K^*(T). \]
	Putting these inequalities together yields the result.
\end{proof} 

This lemma is very similar to \cite[Theorem 8]{CGL}, but the setup here is a bit different: In \cite{CGL} there is the constraint $f(0) = 1$, instead of $f(x) \leq 0$ for $\lvert x\rvert \geq 0$ there is the condition that $r(f) := \inf\{R : f(x) \leq 0 \text{ for } \lvert x \rvert \ge R\}$ is finite, and instead of $\mathcal Z_1$, the functional is 
\[ 
\mathcal{Z}(f) = r(f)  + \frac{2}{r(f)} \int\limits_0^{r(f)} f(x) x dx.\]
By rescaling and renormalizing, we see this gives the same result. However, since $\mathcal Z_n(\cdot)$ is linear, we now only have to solve a single semidefinite program for each $n$. 

As in \cite{CGL}, we use the identity 
\begin{equation}\label{eq:gamma}
\int x^m e^{-\pi x^2} dx = -\frac{1}{2\pi^{m/2+1/2}} \Gamma\Big(\frac{m+1}{2}, \pi x^2\Big),
\end{equation}
where $\Gamma$ is the upper incomplete gamma function, to model  $\mathcal Z_n(f)$ as a linear functional in the matrix entries. This allows us to use the semidefinite programming approach as outlined in Section~\ref{Prelim2} to find the bounds for $n=2,3,4$ from Theorem~\ref{SDP bound}, where we use $d=40$ for all computations. To get the bound $c_n \leq (1+10^{-10})n + 0.3243$, we use the values $f(0)$ and $\mathcal Z_n(f) - nf(0)$ of a (near) optimal function $f \in \mathcal A_{\mathrm{LP}}$ for $\mathcal Z_n(f)$ with $n = 10^4$, also obtained using the same semidefinite programming approach. The proof of the final part of Theorem~\ref{SDP bound} is mentioned in the next section. In order to obtain rigorous proofs we use interval arithmetic to verify the solver output, as is done in \cite{CGL}. The Julia/Nemo/Arb \cite{B, F, J} code to generate the semidefinite programs, to solve them with SDPA-GMP \cite{Na}, and to verify the output using interval arithmetic is included in the arXiv version of this paper.

In order to obtain Corollary~\ref{low degrees}, notice that we have the following relationship between $N_{K,s}$, $N_{K,d}$, and $N_K^*$:
\[ 2N_{K,s}(T) = 2\sum\limits_{\substack{0 < \gamma \le T \\ m_\rho = 1}} 1 \le \sum\limits_{0 < \gamma \le T} \frac{(m_\rho -2)(m_\rho -3)}{m_\rho} = N_K^*(T) - 5N_K(T) + 6N_{K,d}(T). \]
Solving for $N_{K,d}(T)$, using the bounds that Theorem~\ref{SDP bound} gives, and using the bound of $N_{K,s}(T) \ge (\frac{1}{27} + o_K(1))N_K(T)$ for quadratic fields given by \cite{CGG} or the trivial bound of $N_{K,s}(T) \ge 0$ for cubic and quartic fields, we obtain the desired inequalities.

\subsection{Optimizing over functions with bounded support}

As mentioned before, the Cohn-Elkies bound in dimension $1$ is the optimization problem $\inf \{ f(0) : f \in \mathcal A_{\mathrm{LP}}\}$. This means that we can interpret our optimization problems for Dedekind zeta functions as interpolants between the corresponding problem for the zeta function from \cite{CGL} and the Cohn-Elkies sphere packing bound in dimension $1$ from \cite{CE}.

Since we can pack $1$ ball of radius $1/2$ (a unit interval in this case) per unit volume, the optimal center density obviously is $1$. The Cohn-Elkies bound also proves this, for example, via the hat function
\[
H(x) = \begin{cases}
1-|x| & \text{ for } \lvert x \rvert \leq 1,\\
0 & \text{ otherwise},
\end{cases}
\] 
which lies in $\mathcal A_{\mathrm LP}$ and satsifies $H(0)=1$. Since $\mathcal Z_n(H) = n + 1/3$, this provides the proof for the last part of Theorem~\ref{SDP bound}. 

Since $H$ is supported in $[-1,1]$ and $\frac{1}{n}\mathcal Z_n(f)$ is close to $f(0)$ for large $n$, this raises the question whether for large $n$ the optimization problem $\inf \{ \mathcal Z_n(f) : f \in \mathcal A_{\mathrm{LP}}\}$ also has a near optimal solution among the functions supported in $[-1,1]$. 

In particular, since we know that there is no function $f$ of the form $p(x^2) e^{-\pi x^2}$ in $\mathcal A_{\mathrm{LP}}$ with $f(0)=1$ (since $f(1) = 0$ implies $f(n)=\smash{\hat f(n)} = 0$ for all integers $n$ by complementary slackness), the question arises whether we can use optimization over functions supported in $[-1,1]$ to find a function $f$ in $\mathcal A_{\mathrm{LP}}$ satisfying $f(0)=1$ and $2\int_0^1 f(x)x\,dx \approx 0.3243$, so that we can remove the term $10^{-10}$ in the bound $c_n \leq (1 + 10^{-10}) n + 0.3243$ from Theorem~\ref{SDP bound}.

For a quick answer to these questions we suggest the following simple optimization approach based on semidefinite programming, which might be of independent interest. Since $f$ has nonnegative Fourier transform and $\mathrm{supp}(f) \subseteq [-1,1]$, there exists a function $g$ with $\mathrm{supp}(g) \subseteq [-1/2, 1/2]$ such that 
\[
f(x) = g \,*\, g^*(x) = \int_{-\infty}^\infty g(y) g(y-x) \, dy.
\] 
As mentioned by Gallagher \cite{Ga} the existence of such a $g$ follows from the Paley-Wiener theorem and a theorem by Krein \cite[p. 154]{A}. The first theorem says that the nonnegative function $\smash{\hat f}$ is analytic and hence can be approximated by nonnegative cosine polynomials, and the second theorem says that a nonnegative cosine polynomial supported on $[-1,1]$ of the form $f(x)=|g(x)|^2$ for some function $g$ supported on $[-1/2, 1/2]$.

Given $d \geq 1$, we model $f$ via a positive semidefinite matrix $X$ as follows:
\begin{equation}\label{eq:bandlimit}
f(x) = \sum_{i,j=0}^d X_{i,j}\, b_i * b_j^*(x), \quad 
b_i(x) = \begin{cases}
x^i & \text{for } \lvert x \rvert \leq 1/2,\\
0 & \text{otherwise}.
\end{cases}
\end{equation}
Then $f(x) \leq 0$ for $\lvert x \rvert \geq 1$ by construction (since $f$ is supported on $[-1,1]$). And since $X$ can be decomposed as $\sum_{k=1}^{\mathrm{rank}(X)} v_kv_k^{\sf T}$ we have
\[
\hat f(x) = \sum_{k=1}^{\mathrm{rank}(X)} \sum_{i,j=1}^d (v_k)_i (v_k)_j \hat b_i(x) \overline{\hat b_j(x)} = \sum_{k=1}^{\mathrm{rank}(X)} \left|\sum_{i=1}^d (v_k)_i \hat b_i(x)\right|^2 \ge 0.
\] 

Direct computations shows
\[
\hat f(0) = \sum_{i,j=0}^d \frac{1}{(i+1)(j+1)2^{i+j}} A_{i,j} \quad \text{and} \quad f(x) = \sum_{i,j=0}^d X_{i,j} \int_{-T/2+x}^{T/2} y^i (y-x)^j \, dy,
\]
which shows $f(x)$ is a polynomial whose coefficients are linear functions in the entries of $X$, and thus $\hat f(0)$, $f(0)$, and $\int_0^1 f(x) x\, dx$ are linear functionals in the entries of $X$ and therefore 
\[
\inf\big\{\mathcal Z_n(f) : X \succeq 0, \, \hat f(0) = 1\big\}
\]
 is a semidefinite program.  

By solving this problem for $n=1$ and $d=1$, we recover the bound $c_1 \leq 4/3$ from \cite{M}, and by solving it for $n=1$ and $d=40$, we recover the best possible bound $c_1 \leq 1/2+2^{-1/2} \cot(2^{-1/2})$ from \cite{CG} to within $70$ decimals of accuracy. For $n=2,3,4$ the best bounds we can find using this approach are $c_2 \leq 2.3305$, $c_3 \le 3.3315$, and $c_4 \le 4.3320$ (all with $d=40$). As is to be expected, these bounds are not as good as the bounds computed in Theorem~\ref{SDP bound} through optimization over Schwartz functions.

To answer the above questions we set $n=10^4$, which gives use the bound $c_{10^4} \le 10^4 + 0.3333327\dots$. This shows that although the Cohn-Elkies bound in dimension $1$ has a sharp solution $f$ with $\mathrm{supp}(f) \subseteq [-1,1]$, we cannot recover the result $c_n \le (1+10^{-10}) + 0.3242$ from Theorem~\ref{SDP bound} by only considering functions supported in $[-1,1]$. Moreover, if we add the additional constriant $f(0)=1$, then we just get the bound $c_n \leq n + 1/3$ attained by the hat function. So we cannot use optimization over functions supported on $[-1,1]$ to remove the $10^{-10}$ term in Theorem~\ref{SDP bound}. Perhaps the optimal functions for the problems $\inf \{ \mathcal Z_n(f) : f \in \mathcal A_{\mathrm{LP}}\}$ are difficult to construct Schwartz funcions in the same way that the optimal functions for the Cohn-Elkies bound in dimension $8$ and $24$ are difficult to construct Schwartz functions \cite{V,CKMRV}. As for the optimization over unbounded functions, the code to compute the above bounds and to verify the correctness of the bounds using interval arithmetic is included in the arXiv version of this paper.

As a final remark we note that since $f(x)$ is a polynomial in $x$ whose coefficients are linear in the entries of $X$, we can write the constraint $f(x) \leq 0$ for $\lvert x \rvert \ge 1$ as a sum-of-squares constraint (see Section~\ref{Prelim2}). When desired, one can thus optimize over  functions supported in $[-T, T]$ by replacing $1/2$ by $T/2$ in \eqref{eq:bandlimit} and adding these sum-of-squares constraints.

\end{document}